\documentclass{article}

\usepackage[english]{babel}
\usepackage[utf8x]{inputenc}
\usepackage[T1]{fontenc}
\usepackage{amsmath,amssymb,amsthm,mathtools,mathdots,nicefrac,tikz}

\usepackage[a4paper,top=3cm,bottom=2cm,left=3cm,right=3cm,marginparwidth=1.75cm]{geometry}

\usepackage{amsmath}
\usepackage{amsthm}
\usepackage{amssymb}
\usepackage{graphicx}
\usepackage{verbatim}

\usepackage[colorinlistoftodos]{todonotes}

\theoremstyle{definition} % non-italic theorems
\newtheorem{theorem}{Theorem}[section]
\newtheorem{corollary}{Corollary}[section]
\newtheorem{lemma}{Lemma}[section]
\newtheorem{proposition}{Proposition}[section]
\newtheorem{definition}{Definition}[section]
\newtheorem{example}{Example}[section]
\newtheorem{remark}{Remark}[section]

\newcommand{\R}{\mathbb{R}}

\newcommand{\M}{\mathbb{M}}

\title{Moduli dimensions of lattice polygons}
\author{Marino Echavarria, Max Everett, Shinyu Huang, Liza Jacoby, Ralph Morrison,\\Ayush Kumar Tewari, Raluca Vlad, and Ben Weber}
\date{}

\usepackage{graphicx}

\begin{document}

\maketitle

\begin{abstract}
Given a lattice polygon $P$ with $g$ interior lattice points, we associate to it the moduli space of tropical curves of genus $g$ with Newton polygon $P$.  We completely classify the possible dimensions such a moduli space can have.  For non-hyperelliptic polygons the dimension must be between $g$ and $2g+1$, and can take on any integer value in this range, with exceptions only in the cases of genus $3$, $4$, and $7$.  We provide a similar result for hyperelliptic polygons, for which the range of dimensions is from $g$ to $2g-1$.  In the case of non-hyperelliptic polygons, our results also hold for the moduli space of algebraic curves that are non-degenerate with respect to $P$.
\end{abstract}

\section{Introduction}

Tropical geometry studies combinatorial, piece-wise linear analogs of objects from algebraic geometry.  The main objects of study in the two-dimensional case are \textit{tropical plane curves}, which are the tropical versions of algebraic plane curves. A tropical plane curve is a weighted and balanced polyhedral complex in the real plane $\R^2$, pure of dimension $1$. Each tropical plane curve has an associated \textit{lattice polygon}, which is a convex polygon with integer lattice points, called its \emph{Newton polygon}.  The tropical curve is dual to a \textit{subdivision} of this polygon into subpolygons.  Under this duality the vertices of the tropical curve are in bijection with the faces in the subdivision; bounded edges between vertices in the curve correspond to and are perpendicular to interior edges in the subdivision that separate the faces associated to the vertices; and unbounded edges correspond to and are perpendicular to exterior edges of the Newton polygon.  A subdivided lattice polygon and a dual tropical curve are illustrated in Figure \ref{fig:newton-polygon+curve}. Note that while each tropical curve has a unique associated subdivision, the same subdivison can be dual to multiple curves: certainly translation and scaling does not intefere with duality, and usually there is some freedom beyond this in choosing the edge lengths of a tropical curve with prescribed subdivision. 

\begin{figure}[hbt]
    \centering 
    \includegraphics[scale = .4]{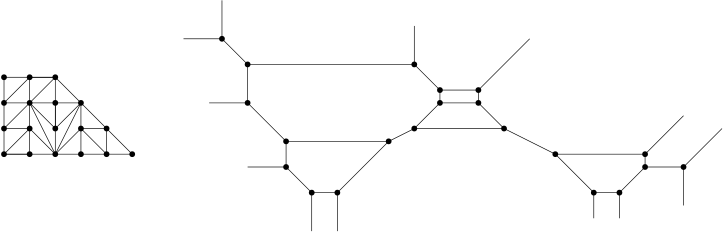}
    \caption{A unimodular triangulation of a lattice polygon, and a smooth tropical curve dual to it.}
    \label{fig:newton-polygon+curve}
\end{figure}

We say tropical curve is \textit{smooth} if its dual subdivision is a unimodular triangulation of its dual Newton polygon -- that is, if the polygon is subdivided into triangles of area $1/2$. In this paper, we will restrict our attention to such smooth tropical plane curves and the triangulations that are dual to them.

The \textit{genus} of a lattice polygon $P$ is the number of lattice points in the interior of $P$. We similarly define the {genus} of a smooth tropical plane curve to be its first Betti number, or equivalently the number of bounded faces; note that the genus of a polygon is the same as the genus of a dual smooth tropical curve.   We say that a polygon is \textit{hyperelliptic} if all its interior points are collinear, and \textit{non-hyperelliptic} otherwise. 

To each tropical plane curve whose dual Newton polygon has genus at least $2$, we associate a metric graph called its \textit{skeleton}. We obtain it by removing any unbounded rays and any leaves, until every vertex has degree at least \(2\); from there we ``smooth over'' the resulting graph, deleting each vertex with degree $2$ and replace its two edges with a single edge.  This skeletonization process is illustrated in Figure \ref{fig:skeleton-example}. Such a skeleton becomes a metric graph after we associate to each of its edge a length equal to the sum of lengths of all initial edges from the tropical curve contributing to it.  We similarly define the genus of the skeleton to be the first Betti number of the graph, or equivalently the number of bounded faces in any planar embedding; this is equal to the genus of any tropical curve giving rise to the skeleton.

\begin{figure}[hbt]
    \centering
    \includegraphics[scale = .25]{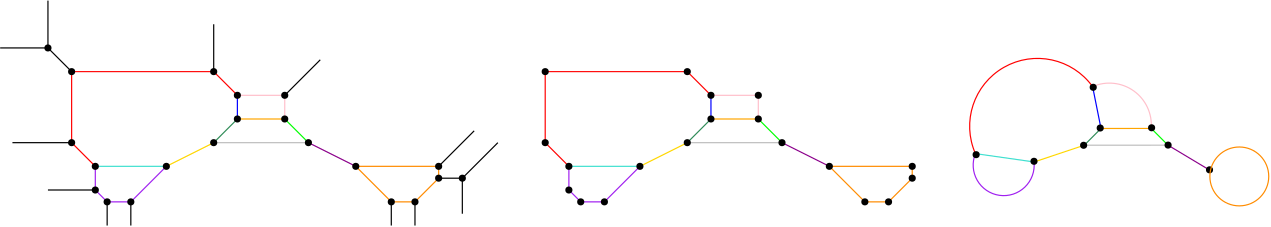}
    \caption{A tropical curve and its genus $5$ skeleton obtained by deleting the ``loose'' ends and then ``smoothing over'' the graph}
    \label{fig:skeleton-example}
\end{figure}

Let $\mathcal{T}$ be a regular unimodular triangulation of a lattice polygon $P$ of genus $g\geq 2$.  Following \cite{bjms}, we define $\M_\mathcal{T}$ to be the closure of the set of all metric graphs arising as the skeleton of some smooth tropical plane curve dual to $\mathcal{T}$. 
We similarly define $\M_P$ to be the closure of the set of all metric graphs arising as the skeleton of some smooth tropical plane curve with Newton polygon $P$. Note that $$\M_{P} = \bigcup_\mathcal{T} \; \M_{\mathcal{T}},$$
where the union is taken over all regular unimodular triangulations $\mathcal{T}$ of $P$.  As discussed in \cite[\S 2]{bjms}, the space $\M_\mathcal{T}$ can be identified with a polyhedral cone in $\mathbb{R}^{3g-3}$ and thus has a well-defined dimension. Intuitively, this is the maximum number of edge lengths we can choose in a skeleton arising from $\mathcal{T}$ before all other lengths are determined.  As a finite union of polyhedral cones, $\M_P$ has dimension equal to the maximum dimension of any $\M_\mathcal{T}$ where $\mathcal{T}$ is a unimodular triangulation of $P$:
$$\dim(\M_P) = \max_\mathcal{T} \; \dim(\M_{\mathcal{T}}).$$
We define the \textit{moduli dimension} of $\mathcal{T}$ to be $\dim(\mathbb{M}_\mathcal{T})$, and the \emph{moduli dimension} of $P$ to be $\dim(\mathbb{M}_P)$.

In addition to introducing the moduli spaces $\M_\mathcal{T}$ and $\M_P$, the work in \cite{bjms} determines the maximum possible moduli dimension of a lattice polygon $P$ of genus $g$:  for a non-hyperelliptic polygon, it is $6$ for $g=3$, $16$ for $g=7$, and $2g+1$ otherwise; and for a hyperelliptic polygon it is $2g-1$.  The fact that these values are upper bounds on moduli dimension follows from work in \cite{cv}, who associate to a lattice polygon $P$ the moduli space $\mathcal{M}_P$ of all algebraic curves that are non-degenerate with Newton polygon $P$.  They determine the maximum possible dimension of such an algebraic space, and properties of tropical geometry imply that $\dim(\mathbb{M}_P)\leq \dim(\mathcal{M}_P)$ \cite[\S 3]{bjms}; thus the same upper bounds hold in the tropical world, and are known to be achieved by the so-called \emph{honeycomb polygons} \cite[\S 4]{bjms}.  It was further shown in \cite{small2017dimensions} that $\dim(\M_P)=\dim(\mathcal{M}_P)$ for any non-hyperelliptic polygon $P$.

Our main contribution in this paper is to determine all possible values of the moduli dimension of a lattice polygon $P$ of genus $g$.  We split into two cases: where $P$ is non-hyperelliptic, and where $P$ is hyperelliptic.

\begin{theorem}\label{theorem:non-hyperelliptic}
There exists a non-hyperelliptic polygon $P$ of genus $g\geq 3$ with $\dim(\mathbb{M}_P)=d$ if and only if
\[l(g)\leq d\leq u(g),\]
where
\[l(g)=\begin{cases}g+1 &\textrm{ if $g\in \{4,7\}$}\\ g&\textrm{ otherwise} \end{cases}\]
and
\[u(g)=\begin{cases}2g &\textrm{ if $g=3$}\\2g+2 &\textrm{ if $g=7$}\\ 2g+1&\textrm{ otherwise.} \end{cases}\]
\end{theorem}

Using the fact that $\dim\left(\mathcal{M}_P\right)=\dim\left(\mathbb{M}_P\right)$ for any non-hyperelliptic polygon $P$, we immediately obtain the following result for algebraic curves.

\begin{corollary}
Theorem \ref{theorem:non-hyperelliptic} still holds if we replace $\mathbb{M}_P$ with $\mathcal{M}_P$.
\end{corollary}

We also have a result similar to Theorem \ref{theorem:non-hyperelliptic} for hyperelliptic polygons.

\begin{theorem}\label{theorem:hyperelliptic}  There exists a hyperelliptic polygon $P$ of genus $g\geq 2$ with $\dim(\mathbb{M}_P)=d$ if and only if
\[g\leq d\leq 2g-1.\]

\end{theorem}

Indeed, we go further by determining $\dim(\mathbb{M}_P)$ for every hyperelliptic polygon $P$, as classifed in \cite{Koelman}.  

%our study of hyperelliptic polygons can be carried further in two ways.  First, we determine $\dim(\mathbb{M}_P)$ for every hyperelliptic polygon $P$, as classifed in \cite{Koelman}.  Next, we explicitly describe the moduli space of all tropical curves arising from a hyperelliptic polygon.  The key tool here is that all such graphs arise from the standard hyperelliptic triangle of genus $g$ by \cite{bjms}.

Our paper is organized as follows. In Section \ref{section:background} we establish necessarily terminology and recall known results, in particular ones concerning the moduli dimension of a triangulation.  In Section \ref{section:non-hyperelliptic} we prove Theorem \ref{theorem:non-hyperelliptic}, with particular care given to determining precisely when the lower bound is achieved in the non-hyperelliptic case.  Finally, in Section \ref{section:hyperelliptic} we prove Theorem \ref{theorem:hyperelliptic} and present other results for hyperelliptic polygons.

\medskip

\noindent \textbf{Acknowledgements.}  This work was carried out as part of the 2020 SMALL REU, supported by NSF Grant DMS-020262 and by Williams College.

\section{Background and Definitions}
\label{section:background}

In this section we recall necessary background on lattice polygons, tropical curves, and dimensions of moduli spaces.

\subsection{Lattice polygons}

Throughout this paper we focus primarily on two-dimensional convex lattice polygons, meaning convex polygons whose vertices have integer coordinates.  
For any such polygon $P$, we let the \textit{interior polygon} $P_\textrm{int}$ be the convex hull of all interior lattice points of $P$.  We call the boundary of $P_\textrm{int}$ the \emph{interior boundary} of $P$, and any lattice points of the interior boundary are called \emph{interior boundary points}.  Note that $P$ is non-hypereliptic precisely when $\dim(P_\textrm{int})=2$.

In the special case of a non-hyperelliptic polygon \(P\), there is a close relationship between \(P\) and \(P_\textrm{int}\).  For a two-dimensional lattice polygon \(Q\), let \(\tau_1,\ldots\tau_n\) denote the $1$-dimensional faces of \(Q\), and let \(\mathcal{H}_{\tau_i}\) denote the half-plane defined by \(\tau_i\), so that
\[Q=\bigcap_{i=1}^n \mathcal{H}_{\tau_i}.\]
We can describe  \(\mathcal{H}_{\tau_i}\) as the set of all points \((x,y)\) satisfying \(a_ix+b_iy\leq c_i\), where \(a_i,b_i,c_i\) are relatively prime integers.  We then let \(\mathcal{H}_{\tau_i}^{(-1)}\) denote the set of all points \((x,y)\) satisfying \(a_ix+b_iy\leq c_i+1\), and define the \emph{relaxed polygon of \(Q\)} to be
\[Q^{(-1)}=\bigcap_{i=1}^n \mathcal{H}_{\tau_i}^{(-1)}.\]
We remark that although \(Q^{(-1)}\) is a polygon, it need not be a lattice polygon.  However, in the case that \(Q=P_{\textrm{int}}\) for some non-hyperelliptic lattice polygon \(P\), we do have that \(P_{\textrm{int}}^{(-1)}\) is a lattice polygon, and is in fact a lattice polygon containing \(P\) with the same set of interior lattice points \cite{cv}.  If \(P=P_\textrm{int}^{(-1)}\), we say that \(P\) is a \emph{maximal} polygon, since it is maximal under containment among all polygons with \(P_\textrm{int}\) as the interior polygon.

Lattice polygons can be mapped to lattice polygons using \emph{unimodular transformations}.  These are affine linear maps that send \(\mathbb{Z}^2\) to itself, and are of the form \(p\mapsto Ap+\overline{v} \), where \(A\) is a \(2\times 2\) integer matrix with determinant \(\pm 1\) and \(\overline{v}\in\mathbb{Z}^2\) is a translation vector.  We say that two lattice polygons \(P\) and \(Q\) are \emph{equivalent} if there exists a unimodular transformation \(t\) such that \(t(P)=Q\).  Several equivalent polygons are illustrated in Figure \ref{figure:equivalent_polygons}; the leftmost can be transformed into the other three using the matrices $\left(\begin{smallmatrix}0&1\\1&0\end{smallmatrix}\right)$, $\left(\begin{smallmatrix}1&-1\\0&1\end{smallmatrix}\right)$, and $\left(\begin{smallmatrix}2&1\\1&1\end{smallmatrix}\right)$.

\begin{figure}[hbt]
    \centering
    \includegraphics[scale =1]{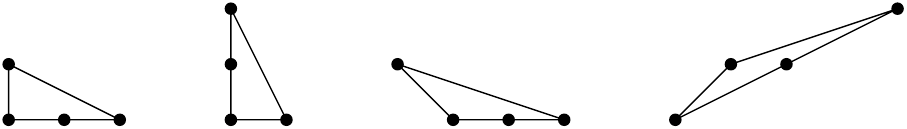}
    \caption{  Four equivalent polygons }
    \label{figure:equivalent_polygons}
\end{figure}

As illustrated by these equivalent polygons, the Euclidean length of the edges of a polygon are not preserved under unimodular transformation; for instance, sending the first to the third turns a vertical edge of length $1$ into a diagonal edge of length $\sqrt{2}$.  This leads us to consider instead the \emph{lattice length} of the edges of lattice polygons, or in general of any line segments with rational slope.  If a line segment $S$ has both endpoints as lattice points, its lattice length $\ell(S)$ is defined to be one less than the number of lattice points on the segment.  We also define any translation of such a line segment to have the same lattice length:  $\ell(S+\overline{v})=\ell(S)$ for any \(\overline{v}\in\mathbb{R}^2\).  Finally, if $S$ is a line segment with rational slope, then there exists some scalar $\lambda$ such that $\lambda S$ has a well-defined lattice length; we then say that $\ell(S)=\frac{1}{\lambda}(\lambda S)$.  Since unimodular transformations send \(\mathbb{Z}^2\) to itself, they preserve the lattice length of line segments.
 Several line segments of various lattice lengths appear in Figure \ref{figure:lattice_lengths}; note that computing the lattice length by counting up lattice points only works when both endpoints are lattice points.

\begin{figure}[hbt]
    \centering
    \includegraphics[scale =1]{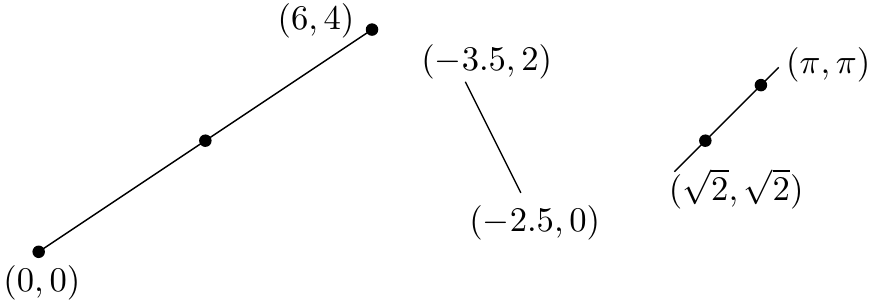}
    \caption{ Three line segments, with lattice points highlighted and endpoints labelled; from left to right, their lattice lengths are \(2\), \(1\), and \(\pi-\sqrt{2}\) }
    \label{figure:lattice_lengths}
\end{figure}

We say that a subdivision $\mathcal{S}$ of $P$ into subpolygons is \emph{regular} if there exists a height function $\omega:(P\cap\mathbb{Z}^2)\rightarrow\mathbb{R}$ such that the lower convex hull of the image of $\omega$, when projected back down onto $P$, yields $\mathcal{S}$. In this case we say that $\omega$ \emph{induces} \(\mathcal{S}\). See \cite{triangulations} for more details.  %Equivalently, a subdivision is regular if it is possible to draw a tropical curve dual to it CITATION.

\subsection{Tropical plane curves}

Tropical plane curves arise from regular subdivisions of lattice polygons.  The usual definition of a tropical plane curve uses the min-plus semiring \(\left(\overline{\mathbb{R}},\oplus,\odot\right)\), where \(\overline{\mathbb{R}}=\mathbb{R}\cup\infty\), \(a\oplus b=\min\{a,b\}\), and \(a\odot b=a+b\).  A polynomial in two variables over this ring is of the form
\[p(x,y)=\bigoplus_{(i,j)} c_{i,j}\odot x^i\odot y^j,\]
which can be written in classical notation as
\[\min_{(i,j)} \left(c_{i,j}+ix+ jy\right),\]
where only finitely many of the \(c_{i,j}\in \overline{\mathbb{R}}\) are distinct from \(\infty\) (the ``zero'' element of the semiring).  The \emph{tropical curve defined by \(p(x,y)\)} is the set of all points in \(\mathbb{R}^2\) where the minimum in \(p(x,y)\) is achieved at least twice.  By \cite[Theorem 3.3.5]{ms}, a tropical plane curve has the structure of a weighted, balanced polyhedral complex of pure dimension \(1\).  As discussed in the introduction, any tropical curve is dual to a regular subdivision of an associated lattice polygon.  This lattice polygon, called the {Newton polygon} of the tropical curve, is the convex hull of the exponent vectors that appear in \(p(x,y)\), i.e., of the lattice points \((i,j)\) such that \(c_{i,j}\neq \infty\).  The subdivision of the Newton polygon \(P\) to which the curve is dual is the regular subdivision induced by the function \(\omega:P\cap\mathbb{Z}^2\rightarrow \mathbb{R}\) defined by \(\omega(i,j)=c_{i,j}\).  As with lattice polygons, we measure lengths of edges on tropical curves using their lattice lengths.

Given a regular unimodular triangulation $\mathcal{T}$ of a lattice polygon $P$ with genus \(g\geq 2\), we recall the construction of the moduli space $\mathbb{M}_\mathcal{T}$; see \cite[\S 2]{bjms} for more details.  First we compute the \emph{secondary cone} $\Sigma(\mathcal{T})\subset \mathbb{R}^{P\cap\mathbb{Z}^2}$, which is the (closure of the) set of all height functions $\omega:P\cap\mathbb{Z}^2\rightarrow\mathbb{R}$ that induce the triangulation $\mathcal{T}$.  Let $E$ denote the set of bounded edges in a tropical curve dual to $\mathcal{T}$.  There exists a  linear map $\lambda:\Sigma(\mathcal{T})\rightarrow\mathbb{R}^E$ that takes a height function $\omega$ and computes from it the lattice length of each edge in $E$ in the tropical curve whose tropical polynomial has the coefficient $\omega(i,j)$ on the term $x^i\odot y^j$.  Thus \(\lambda(\Sigma(\mathcal{T}))\) is, up to closure, the set of all possible edge lengths on a tropical curve dual to \(\mathcal{T}\).  To obtain the lengths on the skeleton of such a tropical curve, we apply another linear map \(\kappa:\mathbb{R}^E\rightarrow\mathbb{R}^{3g-3}\).  This map deletes those edges that do not contribute to the skeleton, and adds up the lengths on any edges that are concatenated in the skeletonization process.  We then define \(\mathbb{M}_\mathcal{T}:=\lambda\circ \kappa(\Sigma(\mathcal{T}))\).  As the image of a polyhedral cone under a linear map, it has a well-defined dimension; recall that we refer to this as the moduli dimension of \(\mathcal{T}\).

Let $\mathcal{T}$ be any subdivision of $P$. A \textit{radial edge} of $\mathcal{T}$ is an edge that connects a boundary point of $P_\textrm{int}$ with a boundary point of $P$ without passing through the interior of $P_\textrm{int}$. As in \cite{small2017dimensions}, we classify each boundary point $p$ of $P_{int}$  based on the structure of $\mathcal{T}$ as follows:

$$p \text{ has } \begin{cases}
\text{Type } 1 & \text{if it is incident to } 1 \text{ radial edge} \\
\text{Type } 2 & \text{if it is incident to at least } 2 \text{ radial edges}\\
& \hspace{5mm} \text{and the boundary points of } P \text{ it connects to are all collinear} \\
\text{Type } 3 & \text{if it is incident to at least } 3 \text{ radial edges}\\
& \hspace{5mm} \text{and the boundary points of } P \text{ it connects to are not all collinear}\end{cases}$$

A triangulation of a polygon, along with the classification of its interior boundary points into Types 1, 2, and 3, is pictured in Figure \ref{fig:types-examples}. We remark that the condition for Type 2 does not require the boundary points of $P$ to be on the same boundary edge; for instance, a Type 2 point could connect to exactly $2$ boundary points that do not lie on a common edge.

\begin{figure}[h]
    \centering
    \includegraphics[scale = .85]{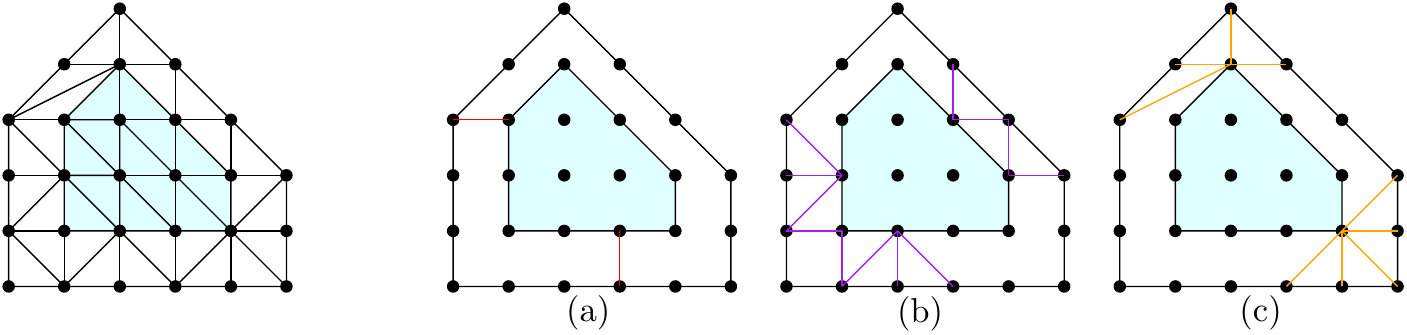}
    \caption{A lattice polygon with its interior polygon, and examples of points of Type (a) $1$, (b) $2$, and (c) $3$ with their radial edges}
    \label{fig:types-examples}
\end{figure}

The following result allows us to compute the moduli dimension of a triangulation based on this classification of interior boundary points.

\begin{theorem}[Theorem 1.2 in \cite{small2017dimensions}]\label{theorem:small2017_formula}
    If $\mathcal{T}$ is a regular unimodular triangulation of non-hyperelliptic polygon $P$ with genus $g$, then
    $$\dim(\M_\mathcal{T}) = g - 3 + \#\{\text{type } 2 \text{ points}\} + 2\cdot \#\{\text{type } 3 \text{ points}\}.$$
\end{theorem}
Let $b_1,b_2$, and $b_3$ denote the numbers of Type 1, Type 2, and Type 3 points in a triangulation. In order to compute $\dim(\mathbb{M}_P)$ for a non-hyperelliptic polygon $P$ it suffices to maximize the value of
\[g - 3 + b_2+2b_3\]
over all regular, unimodular triangulations $\mathcal{T}$ of $P$.  By \cite{small2017dimensions}, it in fact suffices to maximize this value over all unimodular triangulations of $\mathcal{T}$; in other words, there exists a regular triangulation that achieves the maximum.

\begin{example}  Consider the two unimodular triangulations of the same lattice polygon appearing on the left in Figure \ref{figure:g3_moduli_dimension}.  The three interior lattice points are all interior boundary points. The first triangulation has three points of Type 3, while the second has two of Type 3 and one of Type 2.  According to  Theorem \ref{theorem:small2017_formula}, the first triangulation $\mathcal{T}_1$ should have moduli dimension $g-3+b_2+2b_3=3-3+0+2\cdot 3=6$, while the second triangulation $\mathcal{T}_2$ should have moduli dimension $g-3+b_2+2b_3=3-3+1+2\cdot 2=5$. We will verify this by determining which edge lengths are achievable in the skeleton of a tropical curve dual to each triangulation.

\begin{figure}[hbt]
    \centering
    \includegraphics[scale = 1]{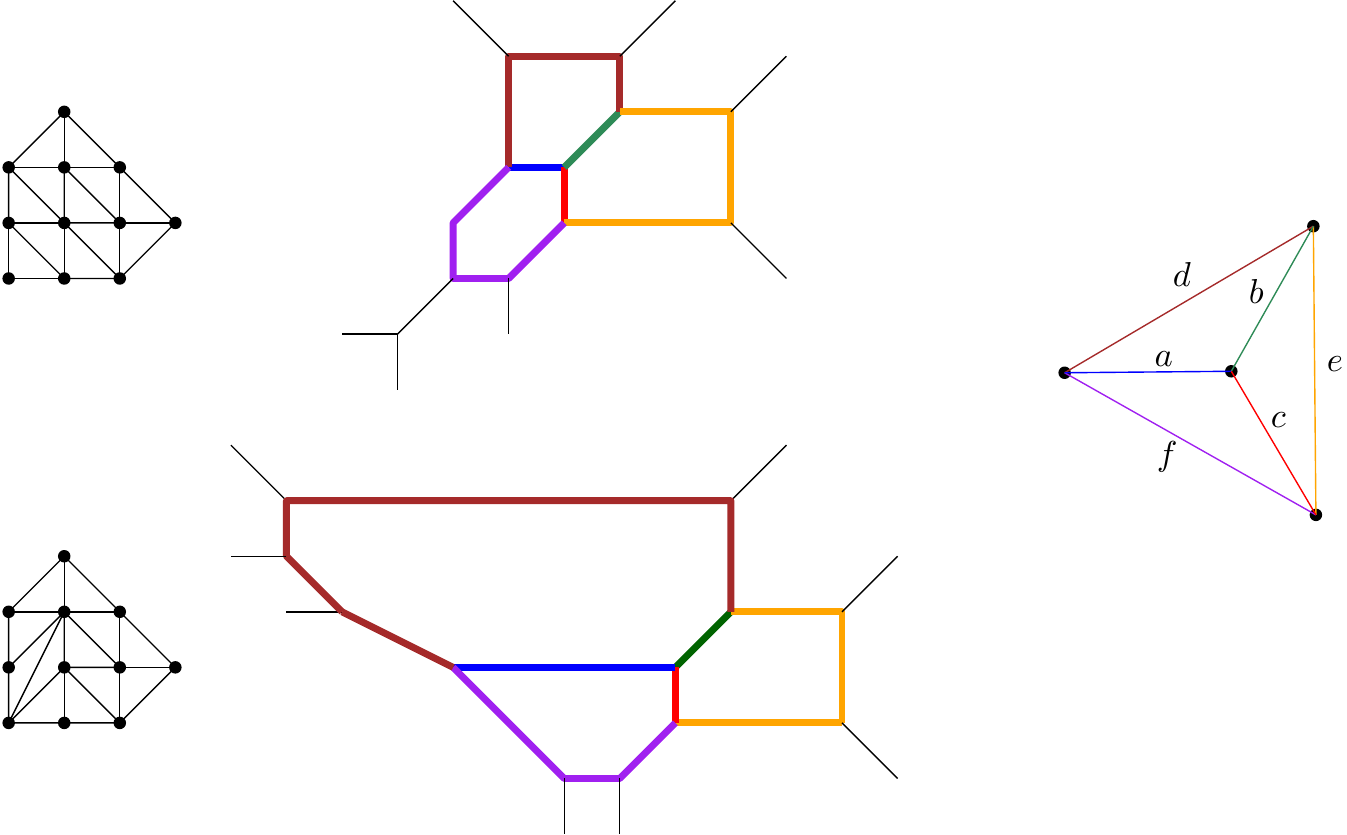}
    \caption{Two regular unimodular triangulations of a lattice polygon of genus $3$, with dual tropical curves that have the same combinatorial type of skeleton}
    \label{figure:g3_moduli_dimension}
\end{figure}

Letting $\Gamma_1$ and $\Gamma_2$ denote tropical curves dual to $\mathcal{T}_1$ and $\mathcal{T}_2$, we note that both have the same combinatorial type of skeleton, namely the complete graph on $4$ vertices illustrated to the right in Figure \ref{figure:g3_moduli_dimension}.  By symmetry we may assume that the lengths of the edges of the skeletons are $a$ through $f$, as labeled.  For $\Gamma_1$, we may choose $a,b,c$ freely.  From there (up to closure), we can choose $d,e,f$ to be any real numbers satisfying $d\geq a+b$, $e\geq b+c$, and $f\geq a+c$.  We find a similar situation with $\Gamma_2$, except that we may not choose $f$: although there are multiple ways to draw the purple edge, we always have that the total lattice length is equal to $a$ due to the slopes of the line segments contributing to the length $f$.  Thus while there are $6$ degrees of freedom in choosing skeletal edge lengths for $\Gamma_1$, there are only $5$ for $\Gamma_2$, as predicted by Theorem \ref{theorem:small2017_formula}.

Since $\dim(\mathbb{M}_P)=\max_{\mathcal{T}}\left\{\dim(\M_\mathcal{T})\right\}$, we have $\dim(\mathbb{M}_P)\geq \dim(\mathcal{T}_1)=6$.  There are several ways to see that $\dim(\mathbb{M}_P)=6$.  Using Theorem \ref{theorem:small2017_formula}, we see that the largest conceivable value of a moduli dimension occurs when all interior boundary points are of Type 3; this gives an upper bound for $g=3$ of $g-3+b_2+2b_3=3-3+0+2\cdot 3=6$, so indeed $\dim(\mathbb{M}_P)\leq 6$.  Alternatively, any skeleton arising from $P$ has $3g-3=6$ edges, so the number of degrees of freedom in choosing edge lengths is certainly no larger than $6$.

\begin{figure}[hbt]
    \centering
    \includegraphics[scale = 1]{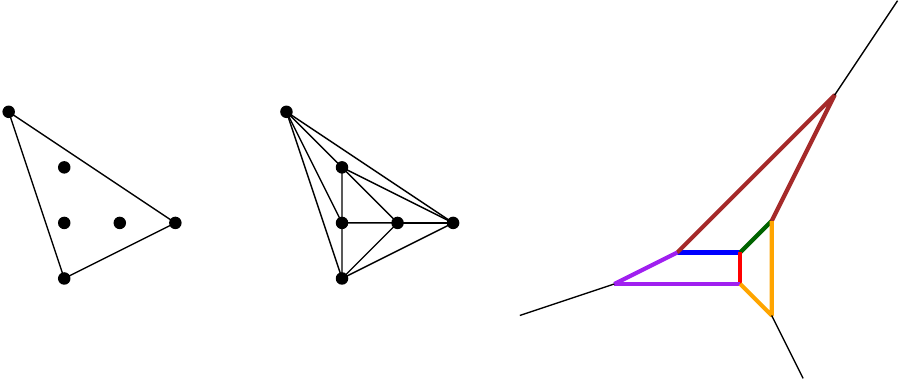}
    \caption{A polygon and a triangulation, both of which have moduli dimension \(3\), along with a tropical curve}
    \label{figure:dim_3_g_3}
\end{figure}

For a more extreme example, we consider another polygon of genus \(3\), pictured in Figure \ref{figure:dim_3_g_3} along with a unimodular triangulation and a dual tropical curve; it has the same skeleton as the curves in Figure \ref{figure:g3_moduli_dimension}.  We may choose the edge lengths \(a,b,\) and \(c\) freely, but this then determines the other three edge lengths in the skeleton, since there are only two contributing edges in the tropical curve for each and so can be drawn in a unique way.  Thus the moduli dimension of the triangulation is \(3\), as predicted by the fact that all three interior points are Type \(2\).  In fact, the moduli dimension of the polygon is equal to \(3\):  no interior lattice point can be connected to more than two boundary points, so three points of Type \(2\) is the best we can do.  This polygon is thus a non-hyperelliptic polygon with moduli dimension equal to its genus; in Proposition \ref{prop:dim_lower_bound_triangle} we will show that this occurs precisely when the polygon has exactly three boundary lattice points.

\end{example}

We close with the following lemma.  It follows immediately from the work in \cite{small2017dimensions}.

\begin{lemma}\label{lemma:nice_triangulation}
Let $P$ be a non-hyperelliptic polygon.  There exists a regular unimodular triangulation $\mathcal{T}$ of $P$ such that
\begin{itemize}
    \item[(i)] $\dim(\mathbb{M}_\mathcal{T})=\dim(\mathbb{M}_P)$,
    \item[(ii)] the boundary of $P_\textrm{int}$ appears in $\mathcal{T}$, and
    \item[(iii)] all edges exterior to $P_\textrm{int}$ are radial edges.
\end{itemize}
\end{lemma}

\section{Moduli dimensions of non-hyperelliptic polygons}
\label{section:non-hyperelliptic}
The main goal of this section is to prove Theorem \ref{theorem:non-hyperelliptic}. Throughout we will use the formula $g-3+b_2+2b_3$ for the moduli dimension of a triangulation.  We remind the reader that to compute the moduli dimension of a polygon $P$, it suffices to maximize the value of $g-3+b_2+2b_3$ over all unimodular triangulations of $P$; we need not worry about regularity by \cite[\S 5]{small2017dimensions}.

We begin with the following proposition, which will help with the lower bound.

\begin{proposition}\label{prop:dim_lower_bound_nontriangle}
If $P$ is a non-hyperelliptic polygon with at least $4$ vertices, then  $\textrm{dim}(\mathbb{M}_P) \geq g+1$.
\end{proposition}

\begin{proof}
Let $v_1, \dots, v_n$ and $w_1, \dots, w_m$  be the  vertices of $P_{\textrm{int}}$ and $P$, respectively, ordered cyclically.  Let $\mathcal{T}$ be a triangulation of $P$ as in Lemma \ref{lemma:nice_triangulation}, with $b_2$ and $b_3$ interior lattice points of Types $2$ and $3$. Now consider the coarser subdivision $\mathcal{T}'$ obtained by ignoring all lattice points besides $v_1, \dots, v_n$ and $w_1, \dots, w_m$.  We may still classify the interior points $v_1, \dots, v_n$ into Types 1, 2, and 3, and we note that their type can only decrease in passing from $\mathcal{T}$ to $\mathcal{T}'$.  Thus letting $b_2'$ and $b_3'$ denote the numbers of Type 2 and Type 3 points in $\mathcal{T}'$, we have $b_2'+2b_3'\leq b_2+2b_3$.  We will show that $b_2'+2b_3'\geq 4$, implying that $\dim(\mathbb{M}_P)=g-3+b_2+2b_3\geq g-3+4=g+1$.

Since no three exterior boundary vertices are collinear, the Type ($1$, $2$, or $3$) of each $v_i$ in $\mathcal{T}'$ is equal to the number of exterior  vertices to which it is connected by a radial edge. Based on our assumption on $\mathcal{T}$ and by \cite[Lemma 4.3]{small2017dimensions}, there are $n+m\geq n+4$ radial edges connecting interior vertices to exterior vertices in $\mathcal{T}'$.  Each radial edge is incident to exactly one interior vertex, so at least $n+4$ radial edges are partitioned among $n$ interior vertices, each of which has at least one radial edge.
    
Assume for the moment that at least one interior vertex $v_i$ is incident to more than $3$ radial edges.  Note that $v_i$ can only be connected to vertices on the two relaxations of the edges of $P_\textrm{int}$ incident to $v_i$; each relaxed edge has at most two vertices, and so $v_i$ is connected to at most (and thus exactly) four boundary vertices.  We note that $v_i$ and $v_{i+1}$ share visibility to an exterior boundary point $w_j$. Therefore, the radial edge $\overline{v_iw_j}$ can be replaced with $\overline{v_{i+1}w_j}$ in the subdivision. If this causes $v_{i+1}$ to have four radial edges, this means that both $v_i$ and $v_{i+1}$ were Type 3 points to begin with, and we have $\dim(\mathbb{M}_P) \geq g+1$. If this situation never arises, then, we can iteratively change our subdivision $\mathcal{T}'$ so that no vertex is connected to more than three radial edges (likewise, we change $\mathcal{T}$ to be any unimodular refinement of $\mathcal{T}$).
% None of the partitions has gained any area since the remaining boundary points of $P_{\textrm{int}}$ will finish the triangulation and will not detract from the number of type 2 or type 3 points. So, we can suppose that no interior vertex has more than three radial edges. 
    
At this point we are partitioning at least $n+4$ radial edges among $n$ interior vertices, so that each gets somewhere between $1$ and $3$ radial edges.  We begin, then, by delegating one radial edge to each $v_i$ so we have at least four remaining radial edges to be distributed.  Each additional edge promotes any Type 1 point to a Type 2 and any Type 2 to a Type 3. There are four or more promotions, each of which contributes $1$ to $b_2' + 2b_3'$, and so we have $b_2' + 2b_3' \geq 4$. This completes the proof.
\end{proof}

We now handle the lower bound in the case of non-hyperelliptic triangles.

\begin{proposition}\label{prop:dim_lower_bound_triangle}
For a non-hyperelliptic triangle $P$, we have
 $\textrm{dim}(\mathbb{M}_P) = g$ if and only if $P$ has exactly three boundary points. Otherwise, $\textrm{dim}(\mathbb{M}_P) \geq  g+1$. 
\end{proposition}

\begin{proof}

We use the same labeling of boundary and interior vertices as the previous proof.  First assume that $P$ has no boundary points besides its three vertices $w_1$, $w_2$, and $w_3$, and let $N$ denote the number of boundary lattice points of $P_\textrm{int}$.  Consider a triangulation $\mathcal{T}$ as in Lemma \ref{lemma:nice_triangulation}.  Viewing $\mathcal{T}$ as a graph, we consider the subgraph $G$ on $N+3$ nodes consisting of the boundaries of $P$ and $P_\textrm{int}$ as well as the radial edges connecting them.  Note the Type of an interior boundary point is determined by how many of $w_1$, $w_2$, and $w_3$ it is connected to.  Since every such point is also connected to two interior lattice points, the Type of such a vertex is thus equal to $\deg(u)-2$.  The sum of the degrees of the interior boundary points is $3N+3$, where $2N$ comes from the inner cycle and $N+3$ comes from the radial edges.  This means
\[3N+3=b_1+2b_2+3b_3+2N,\]
or
\[N=b_1+2b_2+3b_3-3.\]
Since $N$ is also equal to $b_1+b_2+b_3$, it follows that
\[b_1+b_2+b_3=b_1+2b_2+3b_3-3,\]
or
\[3=b_2+2b_3.\]
We therefore have $\dim(\mathbb{M}_P)=\dim(\mathbb{M}_\mathcal{T})=g-3+b_2+2b_3=g$, as claimed.

Now we assume that $P$ has more than $3$ boundary points.  This means that at least one edge of $P$, say $f=\overline{w_1w_2}$, contains a nonvertex boundary point; call it $w$.  Choose $e$ to be an edge of $P_\textrm{int}$ such that $f\subset e^{(-1)}$; by relabelling we may assume that the endpoints of $e$ are $v_1$ and $v_2$.  First, triangulate the point set $\{w_1,w_2,w_3\}\cup\left(P_\textrm{int}\cap \mathbb{Z}\right)$ so that we include the boundary of ${P}_\textrm{int}$; so that there are no edges connecting the boundary of $P$ to itself; so that $v_1$ is connected to both $w_1$ and $w_2$ while $v_2$ is only connected to $w_2$; and so that the triangulation is \emph{fine}, meaning it cannot be subdivided any further without adding new points. By the argument from the previous case, this triangulation has $b_2+2b_3=3$. Note that if $v_2$ is the intersection of the edges $e$ and $e'$, then it can only be connected to lattice points in $e'^{(-1)}$, and thus is either Type 1 (if it is only connected to $w_2$) or Type 2 (if it is connected to another lattice point on $e'^{(-1)}$).  Now add in $w$ to our point set, remove the edge from $v_1$ to $w_2$, and connect both $v_1$ and $v_2$ to $w$.  Whatever type of point $v_1$ was before, it still is after this operation.  If $v_2$ had been Type 1, then it has become Type 2; and if it were Type 2, then it has become Type 3 since $w\notin e'^{(-1)}$.  Thus we have $b_2+2b_3=4$.  Add in any remaining boundary points and refine to a unimodular triangulation, wich will satisfy $b_2+2b_3\geq 4$.  We conclude that $\dim(\mathbb{M}_P)\geq g+1$. %Build a triangulatoin $\mathcal{T}$ that includes the boundary of $P_\textrm{int}$, has no edges connecting the boundary of $P$ to the boundary of $P$, and connects both $v_1$ to both $w$ and $w_1$, and $v_2$ to both $w$ and $w_2$.  We claim that $\dim(\mathbb{M}_\mathcal{T})\geq g+1$. To see this, we construct another subdivision $\mathcal{T}'$ as follows:  remove all boundary points of $P$ besides its vertices and their incident edges from $\mathcal{T}$, and then refine the resulting subdivision to a triangulation.  By the previous case, with only three boundary points, we have $b_2'+2b_3'=3$.  

%We may again choose $\mathcal{T}$ as in Lemma \ref{lemma:nice_triangulation}, and then consider the coarser subdivision $\mathcal{T}'$ that uses no boundary points of $P$ besides its three vertices.  By the previous argument we have that $b_2'+2b_3'=3$.  It follows that $b_2+2b_3\geq 3$.  To see that in fact there must be a strict increase, note that there must be at least one non-vertex boundary point $w$, say on the edge $\overline{w_1w_2}$.  By our assumptions must connect to some interior boundary points
\end{proof}
    
Having established a lower bound on $\textrm{dim}(\mathbb{M}_P)$, we must now determine which values are actually achievable.  We first handle the case of $\textrm{dim}(\mathbb{M}_P)=g$.   
Combining the previous two results, we see that a non-hyperelliptic lattice polygon of genus $g$ has moduli dimension $g$ if and only if has exactly three boundary points.  The following result tells us when such a polygon exists.

\begin{proposition}\label{prop:dim_g_existence}
There exists a non-hyperelliptic polygon of genus $g$ with exactly three boundary points if and only if $g\geq 3$ with $g\notin \{4,7\}$.
\end{proposition}

\begin{proof}
Certainly $g\geq 3$ is a necessary condition, since $g\leq 2$ yields a hyperelliptic polygon.

To prove that the claimed polygons exist, we first handle the cases of $g\equiv 0\mod 3$ and $g\equiv 2\mod 3$.  First we claim that the triangle with vertices at $(1,0)$, $(0,3)$, and $(2k+1,1)$ has exactly three boundary points and genus $3k$ for $k\geq 1$.  The lack of other boundary points follow from the fact that $\gcd(-1,3)=\gcd(2k,1)=\gcd(2k+1,-2)=1$.  For the genus, note that the area of the polygon is $\frac{1}{2}\left|\det\left(\begin{smallmatrix}-1&3\\2k&1\end{smallmatrix}\right)\right|=\frac{1}{2}\left| -1-6k\right|=3k+\frac{1}{2}$.  By Pick's Theorem, the area of the triangle is also equal to $g+\frac{3}{2}-1=g+\frac{1}{2}$, since it has exactly $3$ boundary points.  Since $3k+\frac{1}{2}=g+\frac{1}{2}$ we have $g=3k$ as claimed.  An identical argument shows that the triangle with vertices at $(0,0)$, $(1,3)$, and $(2k+2,1)$ has exactly three boundary points and genus $3k+2$ for $k\geq 1$.   Moreover, all these polygons are non-hyperelliptic;  for instance, each has the points $(1,1)$, $(1,2)$, and $(2,1)$ as interior lattice points.  Thus we have existence whenever $g\equiv 0\mod 3$ or $g\equiv 2\mod 3$.  Several of these polygons are illustrated in Figure \ref{figure:dimension_g}.

\begin{figure}[hbt]
        \centering
        \includegraphics[scale = 0.9]{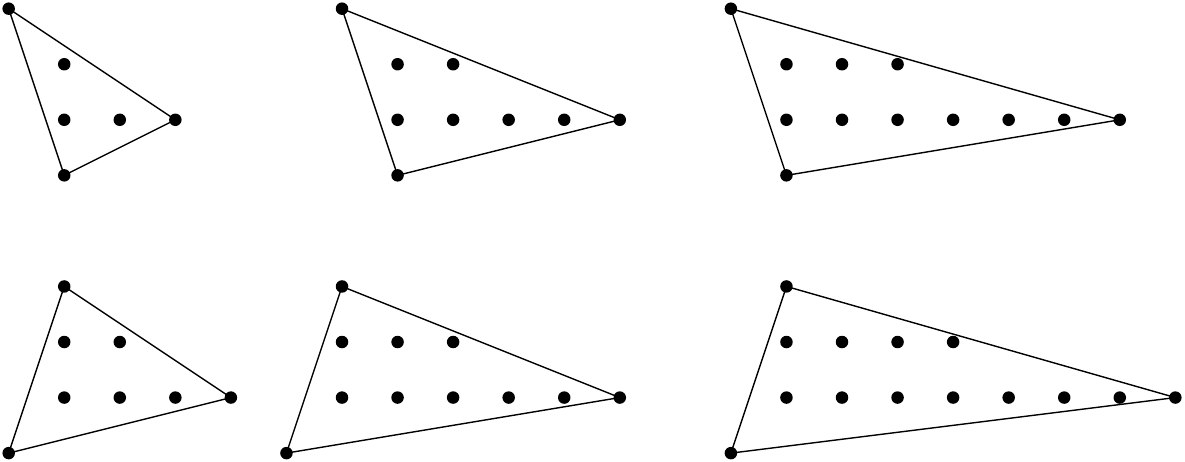}
        \caption{Polygons with moduli dimension equal to their genus}
        \label{figure:dimension_g}
\end{figure}

For the remainder of the proof we will assume $g\equiv 1\mod 3$.  Suppose $T$ is a triangle with exactly three boundary points with such a genus $g$. Since every edge of $T$ has lattice length $1$, after a unimodular transformation we may assume that $T$ has two of its vertices at $(0,0)$ and $(0,1)$.  Letting $(a,b)$ be the other vertex, we perform further unimodular transformations so that  $a,b\geq 0$:  $a\geq 0$ can be achieved with a reflection about a vertical line, possible with a translation, and similarly for $b\geq 0$.  Finally, applying a collection of shearing transformations, we may assume that $0\leq b\leq a$.

We now claim that $a=2g+1$.  To see this, we compute the area of $A$ in two ways.  First, by Pick's Theorem, it is equal to $g+\frac{1}{2}$.  Second, it is equal to $\frac{1}{2}\left|\det\left(\begin{smallmatrix}0&1\\a&b\end{smallmatrix}\right)\right|=\frac{|a|}{2}=\frac{a}{2}$, where we use the fact that $a\geq 0$.  Solving $g+\frac{1}{2}=\frac{a}{2}$, we find $a=2g+1$, as claimed.  Thus if $T$ is a triangle with exactly three lattice points, up to equivalence it has vertices $(0,0)$, $(0,1)$, and $(2g+1,b)$, where $0\leq b\leq 2g+1$ and $\gcd(2g+1,b)=\gcd(2g+1,b-1)=1$. So for each $g$, we simply need to determine whether or not there exists $b$ with $0\leq b\leq 2g+1$ such that $\textrm{conv}((0,0),(0,1),(2g+1,b))$ is non-hyperelliptic.

For now we will assume $g\geq 13$, saving the cases of $g\in\{4,7,10\}$ later.  Since $g\equiv 1\mod 3$, we know $3|(2g+1)$, so we can write $2g+1=3r$ for some $r$.  Since $2g+1$ is odd, so too is $r$. We will choose $b$ so that it satisfies the following two criteria:
\begin{itemize}
    \item $4\leq b\leq \frac{2g+1}{2}$
    \item $\gcd(b,2g+1)=\gcd(b-1,2g+1)=1$
\end{itemize}
First let us argue that this is possible, splitting into cases based on the value of $r$ modulo $3$.  In each case the fact that $b\geq 4$ will follow from the fact that $g\geq 13$.
\begin{itemize}
    \item Assume $r\equiv 0\mod 3$.  Then $r+1$ and $r+2$ do not share any factors with $r$, since they are too close to also have a factor of $3$.  Since $3\nmid r+1$ and $3\nmid r+2$, we can also conclude that $r+1$ and $r+2$ have no prime factors in common with $3r=2g+1$. Thus choosing $b=r+1$ works.
    \item  Assume $r\equiv 1\mod 3$.  Consider $r+3$ and $r+4$.  Since they are within $4$ of $r$, the only factor they could conceivably share with $r$ is $3$, but $3$ does not divide $r$, so they are relatively prime to $r$.  Thus the only conceivable prime factor they should share with $3r=2g+1$ is $3$, but $r+3$ and $r+4$ are $1$ and $2\mod 3$, respectively, so they are relatively prime to $2g+1$, as desired.
    \item  Assume $r\equiv 2\mod 3$.  Consider $r+2$ and $r+3$.  Similar to the previous argument, neither can share a prime factor with $r$, and neither is divisible by $3$, so they cannot share a prime factor with $3r=2g+1$.
\end{itemize}

Now we argue that our choice of $b$ will yield $T$ non-hyperelliptic.  The non-vertical edges of $T$ lie on the lines $y=\frac{b-1}{2g+1}x+1$ and $y=\frac{b}{2g+1}x$, respectively.  Consider the width $w(h)$ of the horizontal cross section of $T$ at height $h\in\mathbb{Z}$.  Since it has no width at height $0$, we have $h(0)=0$; for other heights, we're going from the line $y=\frac{b-1}{2g+1}x+1$ to $y=\frac{b}{2g+1}x$.  Thus the width is $x_2-x_1$, where $h=\frac{b-1}{2g+1}x_1+1$ and $h=\frac{b}{2g+1}x_2$.  Solving for $x_1$ and $x_2$ gives $x_1=\frac{(h-1)(2g+1)}{(b-1)}$ and $x_2=\frac{h(2g+1)}{b}$.  Thus we have $w(h)=x_2-x_1=\frac{h(2g+1)}{b}-\frac{(h-1)(2g+1)}{(b-1)}=(2g+1)\left(\frac{h}{b}-\frac{h-1}{b-1}\right)=\frac{(2g+1)(b-h)}{b(b-1)}$.

First note that since $b\leq\frac{2g+1}{2}$, we have $w(1)=\frac{(2g+1)(b-1)}{b(b-1)}=\frac{2g+1}{b}\geq 2$.  This means that the points $(1,1)$ and $(1,2)$ are interior to $T$.  Next we will argue that $w(2)\geq 1$, which will imply that there is an interior lattice point of $T$ at height $2$.  This means we wish to show that $\frac{(2g+1)(b-2)}{b(b-1)}\geq 1$, or equivalently that $(2g+1)(b-2)\geq b(b-1)$.  Since $b\geq 4$, we have $b-2\geq b/2$, so it suffices to show $(2g+1)\frac{b}{2}\geq b(b-1)$. This occurs when $2g+1\geq 2(b-1)$, which certainly holds since $2g+1\geq 2b$.  Thus there are at least two interior points at height $1$ and at least one interior point at height two, implying that $T$ is non-hyperelliptic.

We have now shown that there exists a non-hyperelliptic triangle of genus $g$ with exactly three boundary points whenever $g\equiv 0\mod 3$, or $g\equiv 2 \mod 3$, or $g\geq 13$ with $g\equiv 1 \mod 3$. It remains to handle the cases of $g=4$, $g=7$, and $g=10$.  By our previous work, any triangle of genus $4$ with exaclty three interior lattice points is equivalent to one with vertices at $(0,0)$, $(0,1)$, and $(9,b)$ where $0\leq b\leq 9$ and $\gcd(b,9)=\gcd(b-1,9)=1$.  It follows that $b\in \{2,5,8\}$.  However, all three choices of $b$ yield a hyperelliptic polygon.  A similar phenomenon occurs for $g=7$, where the third vertex is $(15,b)$ where $b\in\{2,8,14\}$, again yielding only hyperelliptic triangles.  In the case of $g=10$, we do manage to find a non-hyperelliptic triangle, for instance one with vertices at $(0,0)$, $(0,1)$, and $(21,5)$. This completes the proof.
\end{proof}

We are now ready to prove our main theorem for non-hyperelliptic polygons.  Recall the definitions \[l(g)=\begin{cases}g+1 &\textrm{ if $g\in \{4,7\}$}\\ g&\textrm{ otherwise} \end{cases}\]
and
\[u(g)=\begin{cases}2g &\textrm{ if $g=3$}\\2g+2 &\textrm{ if $g=7$}\\ 2g+1&\textrm{ otherwise.} \end{cases}\]

\begin{proof}[Proof of Theorem \ref{theorem:non-hyperelliptic}]
 Let $P$ be a non-hyperelliptic polygon of genus $g$.  The upper bound $\dim(\mathbb{M}_P)\leq u(g)$ follows from \cite{bjms}.  The lower bound $\ell(g)\leq \dim(\mathbb{M}_P)$ follows from a combination of Propositions \ref{prop:dim_lower_bound_nontriangle}, \ref{prop:dim_lower_bound_triangle}, and \ref{prop:dim_g_existence}. 
 It remains to show that all values of $d$ with $\ell(g)\leq d \leq u(g)$ are in fact the moduli dimension of some non-hyperelliptic polygon of genus $g$.  First we will show that all values of $d$ between $g+1$ and $2g+1$ are achieved.
 
For $g=2h$ even, set $P$ to be the rectangle $\textrm{conv}\left((0,0),(0,3),(h+1,0),(h+1,3)\right)$; for $g=2h+1$ odd, set $P$ to be the trapezoid $\textrm{conv}\left((0,0),(0,3),(h+3,0),(h,3)\right)$. The polygon $P$ has genus $g$, and by \cite[\S 4]{bjms} and \cite{small2017dimensions} we know $\dim(\mathbb{M}_P)=2g+1$.  Define $Q$ to be a subpolygon of $P$ with the same interior polygon, in the even case as $\textrm{conv}\left((1,0),(0,3),(h+1,1),(h+1,2)\right)$ and in the odd case as $\textrm{conv}\left((1,0),(0,3),(h+2,1),(h+1,2)\right)$.  These polygons $P$ and $Q$ for the even and odd cases are illustrated in Figure \ref{figure:P_and_Q}.  Because $Q$ only has $4$ boundary points (all vertices), an optimal triangulation will only yield moduli dimension $g+1$; for instance, a parallel argument to the first half of the proof of Proposition  \ref{prop:dim_lower_bound_triangle} will show this. Thus we have $\dim(\mathbb{M}_Q)=g+1$.  Consider a sequence of convex lattice polygons $P=P_0\supsetneq P_1\supsetneq P_2\supsetneq\cdots\supsetneq P_k=Q$, where $P_i$ has exactly one more lattice point than $P_{i+1}$.  Since $P_\textrm{int}=Q_\textrm{int}$, all these polygons have genus $g$.  As argued in \cite[\S 5]{small2017dimensions}, the moduli dimension can drop by at most $1$ in going from $P_i$ to $P_{i+1}$.  Since $\dim(\mathbb{M}_P)=2g+1$ and $\dim(\mathbb{M}_Q)=g+1$, every integer between $g+1$ and $2g+1$ must be equal to $\dim(\mathbb{M}_{P_i})$ for some~$i$.

\begin{figure}[hbt]
        \centering
        \includegraphics[scale = 0.9]{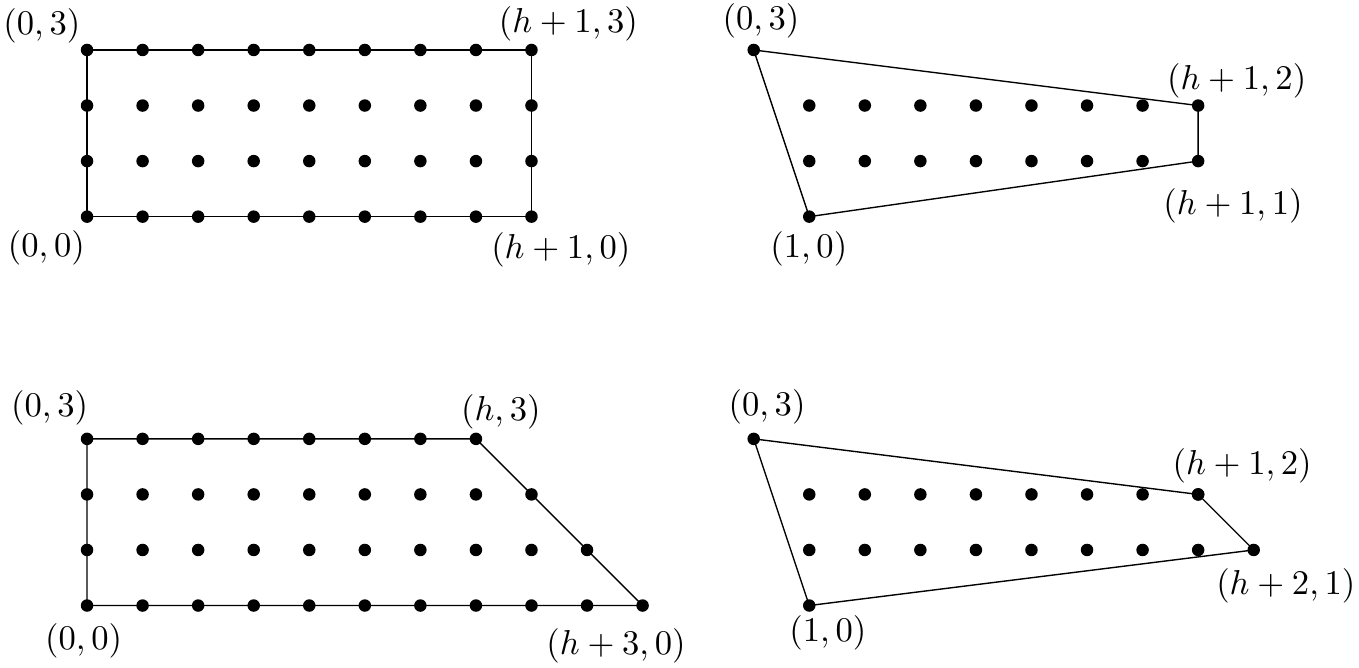}
        \caption{The polygon $P$ on the left and $Q$ on the right, in the even case (top) and the odd case (bottom)}
        \label{figure:P_and_Q}
\end{figure}

 We now have that all values of $d$ between $g+1$ and $2g+1$ are achieved as the moduli dimension of some polygon of genus $g$; the same is true for $d=g$ when $g\notin \{4,7\}$ by Proposition \ref{prop:dim_g_existence}.  This gives us all values of $d$ with $\ell(g)\leq d\leq u(g)$ when $g\neq 7$.  For $g=7$, we also need a polygon achieving the moduli dimension $2g+2$; this is furnished by \cite[Theorem 1.1]{bjms}, completing the proof.
\end{proof}

The fact that every intermediate moduli dimension between the upper and lower pounds is achieved is not something we can take for granted.  For instance, if we restrict our attention to maximal non-hyperelliptic polygons, it is no longer the case that all intermediate values are achieved as a moduli dimension.

\begin{proposition}
    Let $g\geq 20$.  If $g$ is even, there does not exist any maximal non-hyperelliptic polygon of dimension $2g-1$.  If $g$ is odd, there does not exist any maximal non-hyperelliptic polygon of dimension $2g$.
\end{proposition}

\begin{proof}
The key polygons $P$ of genus $g$ in this proof will be those whose interior polygon $P_\textrm{int}$ has genus $g^{(1)}=0$.  For $g\geq 7$, we know that $P_\textrm{int}$ is a trapezoid $T_{a,b}$ with vertices at $(0,0)$, $(0,1)$, $(b,0)$, and $(a,1)$ where $0\leq a\leq b$ and $b\geq 1$.   By \cite[Proposition 4.1]{panoptigons}, $T_{a,b}$ is an interior polygon if and only if  $a\geq \frac{1}{2}b-1$, or equivalently $a\geq\frac{g-2}{3}$.  For such a $T_{a,b}$, the vertices of $T_{a,b}^{(-1)}$ are $(-1,-1)$, $(-1,2)$, $(2b-a+1,-1)$, and $(2a+1-b,2)$.

Let us compute the moduli dimension of $T_{a,b}^{(-1)}$. By \cite{small2017dimensions}, to do this we may connect the interior vertices to their corresponding boundary vertices, add in the boundary of the interior polygon, and produce a ``zig-zag'' pattern on the resulting width-$1$ polygons.  For the interior points at height $0$, this yields two points of Type 3 and $b-2$ points of Type 2 (contributing $b-2+2\cdot 2=b+2$ to $b_2+2b_3$).  Before implementing the zig-zag pattern for the top strip, there are two points (the interior vertices) at height $1$ already of Type $2$.  Each lattice point at height $2$ beyond $(0,0)$ allows us to promote one point at that height by one type, increasing $b_2+2b_3$ by one, until we get to the point that every interior point has been maximized as Type 3 for a vertex or Type 2 otherwise (yielding $b_2+2b_3=g-4+2\cdot 4=g+4$).  There are $2a-b+2$ lattice points at height $2$, so the final value of $b_2+2b_3$ is
\[\min\{g+4,b+2+2a-b+1\}=\min\{g+4,2a+3\}\]
This yields a moduli dimension of
\[\min\{2g+1,g+2a+2\}.\]
Recalling that $\frac{g-2}{3}\leq a\leq \frac{g}{2}$, the only value of $a$ for which the minimum is $2g+1$ is $a=\left\lfloor\frac{g}{2}\right\rfloor$.  For $a=\left\lfloor\frac{g}{2}\right\rfloor-1$, the moduli dimension is $g+2a+3=g+2\left\lfloor\frac{g}{2}\right\rfloor$, which equals $2g$ for $g$ even and $2g-1$ for $g$ odd.  Decreasing $a$ from here causes the moduli dimension to drop by $2$ at a time, until reaching $g+\left\lceil\frac{g-2}{3}\right\rceil+2$.  In particular, for $g$ even, none of these polygons have moduli dimension $2g-1$, and for $g$ odd none have moduli dimension $2g$.

It remains to show that no other maximal non-hyperelliptic polygon has moduli dimension $2g-1$ or $2g$.  Any other such polygon has $g^{(1)}\geq 1$.  Let $r^{(1)}$ denote the number of interior boundary points, so that  $g=r^{(1)}+g^{(1)}$.   By \cite{scott} we know that the number of boundary points of a convex lattice polygon of positive genus is at most twice its genus plus seven, so $r^{(1)}\leq 2g^{(1)}+7$.  Thus $g=r^{(1)}+g^{(1)}\leq 3g^{(1)}+7$, which implies $g^{(1)}\geq (g-7)/3$.  Since we've assumed $g\geq 20$, we have $g^{(1)}\geq (20-7)/3>4$, so $g^{(1)}\geq 5$.  By \cite[Corollary 10.6]{cv}, we have that $\dim(\mathcal{M}_P)\leq 2g+3-g^{(1)}$ for any maximal non-hyperelliptic polygon $P$.  This also serves as an upper bound on $\dim(\mathbb{M}_P)$.  Since $g^{(1)}\geq 5$, we have $\dim(\mathbb{M}_P)\leq 2g+3-5=2g-2$.  Thus no maximal non-hyperelliptic polygon with $g^{(1)}\geq 1$ and $g\geq 20$ has moduli dimension $2g-1$ or $2g$.  This completes the proof.
\end{proof}

%\begin{proposition}
%    Let $N\geq 1$ be an integer.  There exists $g=g(N)$ such that the set of integers \[\{\dim(\mathbb{M}_P)\,|\,\textrm{$P$ is maximal of genus $g$}\}\] has a gap of size at least $N$
%\end{proposition}

%\begin{proof}
  %  Let $P$ be a maximal non-hyperelliptic lattice polygon of genus $g$ whose interior polygon has genus $g^{(1)}$.  First consider the case that $g^{(1)}=0$.  For $g\geq 7$, this means that we may assume $P_\textrm{int}$ is a trapezoid of height $1$, with vertices at $(0,0)$, $(0,1)$, $(b,0)$, and $(a,1)$ where $0\leq a\leq b$ and $b\geq 1$.  By the structure of such a $P$, all $b$ points along the lower interior row can be optimized to points of Type 3 (for the two interior vertices) or of Type 2 (for the other $b-2$ points).  This certainly yields
  %  \[\dim(\mathbb{M}_P)\geq g-1+(b-2)+2\cdot 2\geq g+2+\frac{g}{2}=\frac{3g+1}{2}.\]

    %By PANOPTIGONS, we must have $a\geq \frac{1}{2}b-1$, or equivalently $a\geq\frac{g-2}{3}$.  In such a polygon $P$, the $b$ interior lattice points on the lower row can be made points of Type 3 (if interior vertices) or Type 2 (if not) in a unimodular triangulation of $P$; the limiting factor on dimension comes in for the $a$ interior points on the upper row. Note that 

 %   By \cite[Corollary 10.6]{cv}, we have that $\dim(\mathcal{M}_P)\leq 2g+3-g^{(1)}$, where $P$ is a polygon of genus $g$
%\end{proof}

Further investigation into the achievable moduli dimensions of maximal non-hyperelliptic polygons would be an interesting direction for future research.

%{\color{orange} Could we find a family of triangulations which achieves the lowest dimension of $g$? As in the case of honeycomb triangulations which achieve the maximum dimension of $2g + 1$, maybe this could be a direction of future research}
%{\color{blue}  I think it's all triangulations that only use three boundary points (where ``use'' means we only have )}

\section{Moduli dimensions of hyperelliptic polygons}
\label{section:hyperelliptic}

We now move on to the case of hyperelliptic polygons, which have all interior lattice points collinear.  These are somewhat simpler to describe than non-hyperelliptic polygons, and in fact a complete classification of hyperelilptic polygons of genus $g\geq 2$ was carried out in \cite{Koelman}, and is also presented in \cite{movingout}.  We recall that result here.

\begin{theorem}[\cite{Koelman}] \label{theorem:hyperelliptic_classification} If $P$ is a hyperelliptic polygon of genus $g\geq 2$, then it is equivalent to precisely one of the following polygons, sorted into three classes.

\begin{enumerate}
    \item Class 1: 
    \begin{center}
        \includegraphics[scale = 0.2]{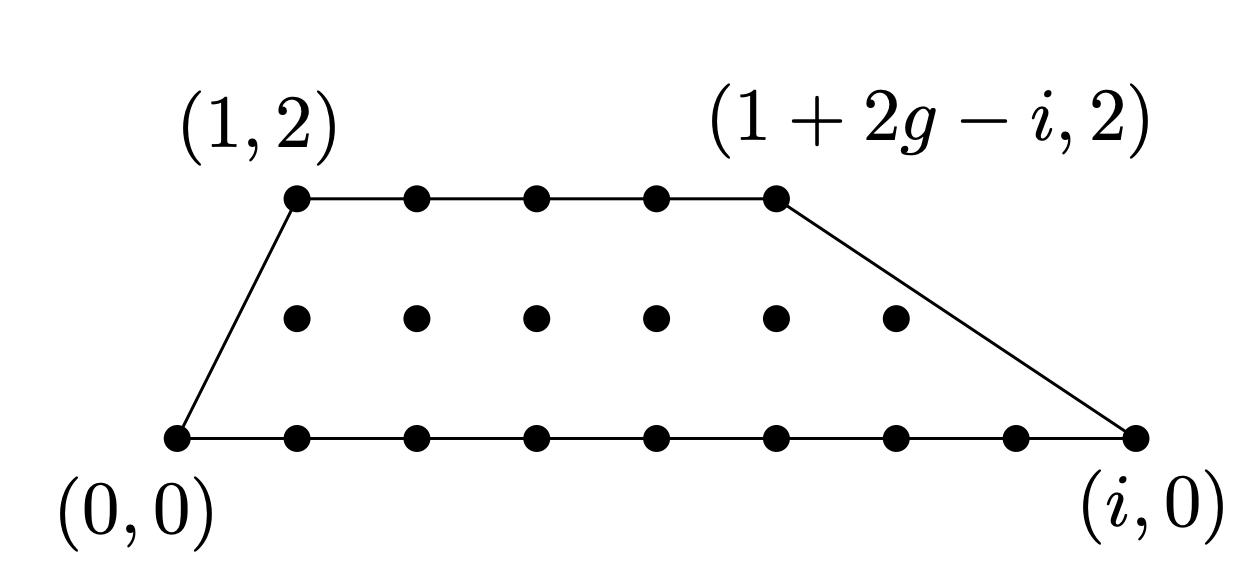}
    \end{center}
    where $g \leq i \leq 2g$
    
     \item Class 2: 
    \begin{center}
        \includegraphics[scale = 0.2]{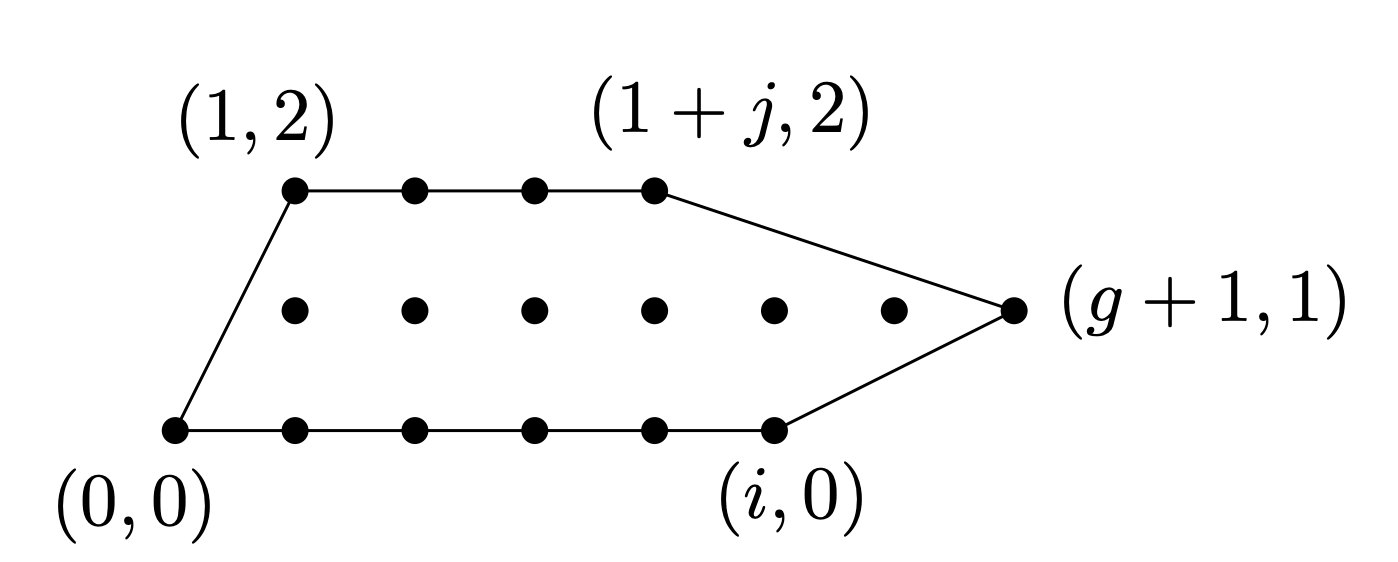}
    \end{center}
     where (a) $0 \leq i \leq g$ and $0 \leq j \leq i$; or (b) $g < i \leq 2g+1$ and $0 \leq j \leq 2g-i + 1$
    
     \item Class 3: 
    \begin{center}
        \includegraphics[scale = 0.2]{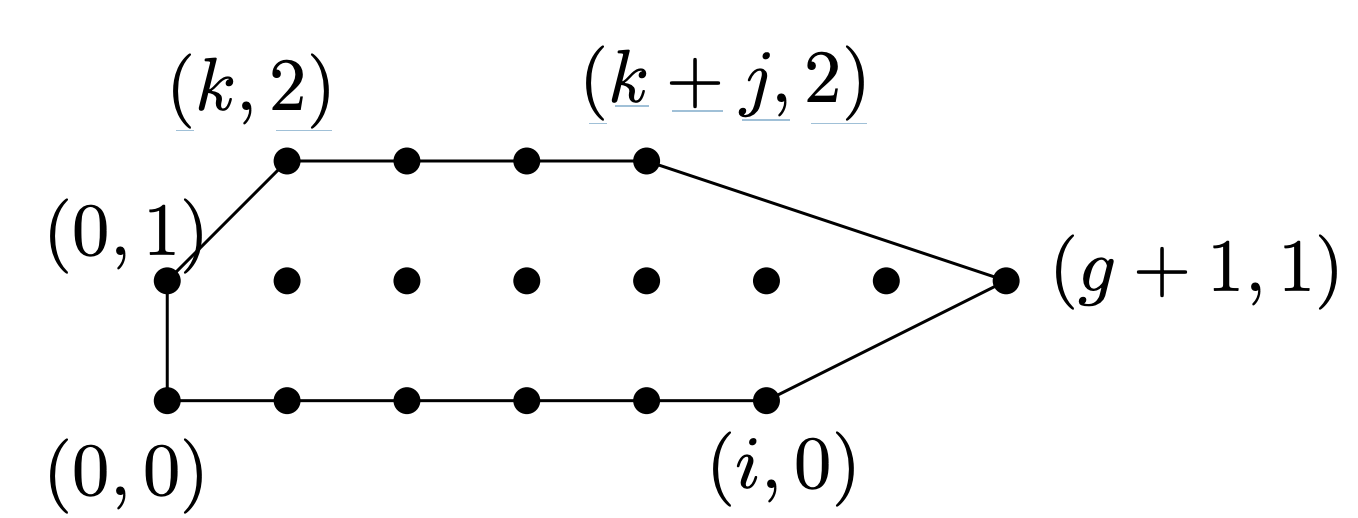}
    \end{center}
    where (a) $0 \leq k \leq g+1$ and $0 \leq i \leq g+1−k$ and $0 \leq j \leq i$; or (b) $0 \leq k \leq g+1$ and $g+1−k < i \leq 2g+2−2k$
and $0 \leq j \leq 2g − i − 2k + 1$
\end{enumerate}

\end{theorem}

  \begin{figure}[hbt]
        \centering
        \includegraphics[scale = 1]{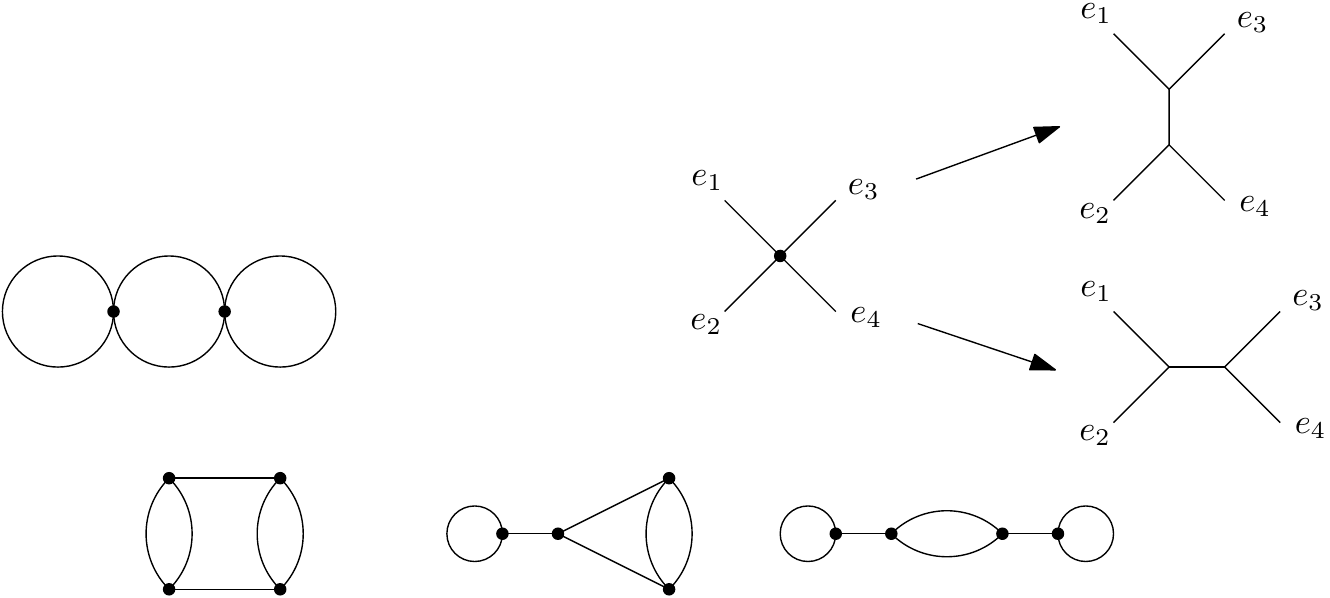}
        \caption{The starting $4$-regular graph, the two resolution choices, and the resulting trivalent chains of genus $3$}
        \label{figure:genus_3_chains}
    \end{figure}

Things are also simpler on the side of dual tropical plane curves.  As discussed in \cite[\S 6]{bjms}, such a curve's skeleton is a \emph{ chain}.  A chain of genus $g$ can be constructed as follows.  Start with a line segment on $g-1$ vertices.  Duplicate each edge to a bi-edge, and attach two loops to the two endpoints.  This gives us a $4$-regular graph. For each $4$-valent node, we resolve it into two $3$-valent nodes in one of two ways, one giving shared edge between two bounded faces and one giving a bridge connecting two bounded faces.  The starting $4$-regular graph, the two allowed resolutions for each node, and the resulting chains for genus $3$ are illustrated in Figure \ref{figure:genus_3_chains}.

We set the following labelling convention for the edge lengths of a chain, illustrated in Figure \ref{figure:chain_labels}.  We will assume our chain has come from an embedding of a tropical curve, and so has a well defined orientation of its bounded faces from left to right. The leftmost edge length is called $e$, and the rightmost $f$.  For the $i^{th}$ bounded face where $2\leq i\leq g-1$, we let $u_i$ (respectively $\ell_i$) denote the length of the upper (respectively lower) edge bounding that face.  Finally, if the $i^{th}$ and $(i+1)^{th}$ cycles share a bridge, call its length $b_{i,i+1}$ (and by convention set $h_{i,i+1}=0$), and if they instead share an edge, call its length $h_{i,i+1}$ (and by convention set $b_{i,i+1}=0$)

  \begin{figure}[hbt]
        \centering
        \includegraphics[scale = 1]{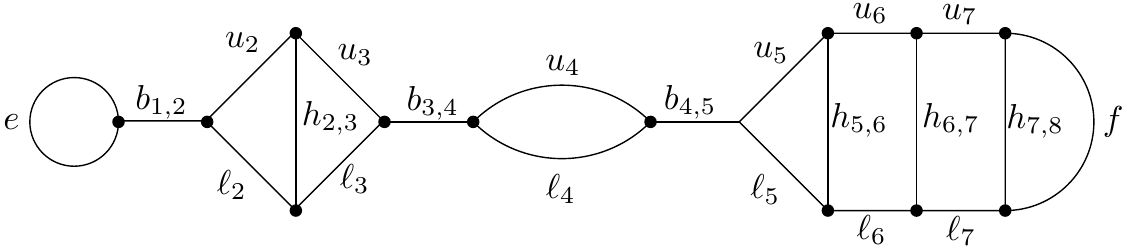}
        \caption{The edge length labels for a chain}
        \label{figure:chain_labels}
    \end{figure}
    
Given a unimodular triangulation $\mathcal{T}$ of a hyperelliptic Newton polygon $P$, we wish to consider the constraints on the edge lengths of a skeleton $\Gamma$ of a dual tropical curve.  There are a few immediate constraints (or lack thereof) on the edge lengths of $\Gamma$.  First, all edge lengths must be positive (or non-negative, after we take closures).  Second, we have $u_i=\ell_i$ for all $i$ by \cite[Lemma 2.2]{morrison-hyperelliptic}.  Third, there are no restrictions on any nonzero $b_{i,i+1}$'s; this is a general property of bridges in tropical plane curves.  Similarly, if $e$ (or $f$) is the length of a loop incident to a bridge, then there are no restrictions on it.

The further inequalities are more subtle, and depend on the structure of the triangulation $\mathcal{T}$.  Without loss of generality, we assume that $P$ is of a form from Theorem \ref{theorem:hyperelliptic_classification}, with interior lattice points at $(i,1)$ for $1\leq i\leq g$.  For each $i$, let $NW(i)$ denote the $x$-coordinate of the leftmost lattice point at height $2$ connected to $(i,1)$; $NE(i)$ that of the rightmost such point at height $2$; and similarly for $SW(i)$ and $SE(i)$ at height $0$.

Assume for the moment that $2\leq i\leq g-1$.  We claim that the only further constraint involving any of $u_i$, $h_{i,i-1}$, and $h_{i,i+1}$ is the following:
\[(2i-NE(i)-SE(i))u_i\leq h_{i-1,i}-h_{i,i+1} \leq (2i-NW(i)-SW(i))u_i.\]
This is proven in \cite[\S 3.6]{morrison_thesis}, and comes down to the fact that $2i-NW(i)-SW(i)$ is the difference of the slopes coming out of the $h_{i-1,i}$ edge to start the $i^{th}$ bounded face; by the convexity of this loop, the longest possibility for $h_{i,i+1}$ compared to $h_{i-1,i}$ is spend as much of the width $u_i$ on these edges in a tropical curve dual to $\mathcal{T}$.  This gives the upper bound; the lower bound follows from a symmetric argument.  Note that we may not necessarily choose $u_i$ and $h_{i-1,i}$ freely, and then choose $h_{i,i+1}$ to satisfy these bounds:  we must also be sure to satisfy the condition that all lengths are non-negative.

%The last consideration comes in the edge lengths $e$ and $f$.  We will only consider $e$; $f$ is handled symmetrically.  First, if the $e$ edge is a loop, there are no constraints on it.  Now consider the case where $(1,1)$ is connected to $(2,1)$.  We define $NW(1)$, $NE(1)$, $SW(1)$, and $SE(1)$ in parallel to before.  If $NE(1)+SE(1)-2>0$, then the first bounded face ``opens up'' away from the shared edge with the second bounded face, and so there is no constraint on $\ell$ except $\ell\geq h_{1,2}$.  If 

%In this section we will prove Theorem \ref{theorem:hyperelliptic} for hyperelliptic polygons.  In contrast to the non-hyperelliptic case studied in \cite{small2017dimensions}, there does not exist in the literature a combinatorial formula for computing $\dim(\mathbb{M}_\mathcal{T})$, where $\mathcal{T}$ is a triangulation of a hyperelliptic polygon.  Thus our first goal will be to create such a formula.

% Let $P$ be a hyperelliptic polygon of genus $g$, so that $P_\textrm{int}$ is a line segment.  Throughout this section, we will assume that $P$ is contained in the horizontal strip with $0\leq y\leq 2$ We call the interior lattice points that are the two endpoints $P_\textrm{int}$ the \emph{end interior points}. We call the other interior lattice points \emph{middle interior points}. Note that by the structure of a hyperelliptic polygon, in any unimodular triangulation each middle interior point will be connected to at least one boundary point at height $0$ and one boundary point at height $2$.

Now we consider the edge $e$.  For simplicity, we will assume that no edge in $\mathcal{T}$ connects the boundary to the boundary without cutting through the interior line segment; if there does exist such an edge, then we can iteratively pare off the triangles formed by such an edge and the boundary of $P$ and obtain a triangulation $\mathcal{T}'$ of another hyperelliptic polygon $P'$ of the same genus, realizing the exact same metric graphs.

By Theorem \ref{theorem:hyperelliptic_classification}, we know that the left end of hyperelliptic polygon has one of two shapes, either Shape $A$ (for Classes 1 and 2) or Shape $B_k$ (for Class 3).  These shapes are illustrated in Figure \ref{figure:two_end_types}.  We say that $\mathcal{T}$ \emph{has form $A(m)$ with respect to $e$} (resp. form $B_k(m)$ with respect to $e$) if the end has shape $A$ (resp. $B_k$) and if $(1,1)$ is incident to exactly $m+1$ edges in $\mathcal{T}$:  one connecting it to $(2,1)$, and $m$ connecting it to boundary points.  By our assumption on $\mathcal{T}$, for Shape $A$ we know $e$ is connected to at least $(0,0)$ and $(1,2)$, and for Shape $B_k$ to $(0,0)$, $(0,1)$, and $(k,2)$; thus for shape $A$ we have $m\geq 2$, and for shape $B_k$ we have $m\geq 3$.  Depending on the form of triangulation, we will give the precise constraints on $e$.

 \begin{figure}[hbt]
        \centering
        \includegraphics[scale = 1]{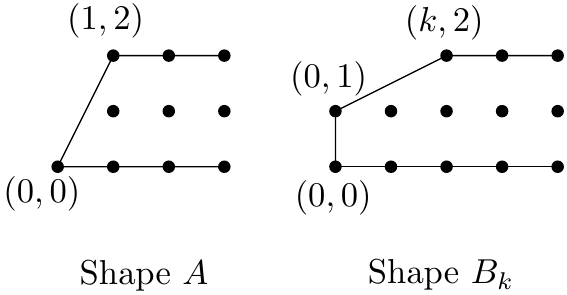}
        \caption{Two shapes of the left end of a hyperelliptic polygon}
        \label{figure:two_end_types}
    \end{figure}
    
First assume that $P$ has shape $A$.  Several forms of the triangulation $\mathcal{T}$ appear in Figure \ref{figure:A_forms}, along with the face dual to $(1,1)$ in a tropical curve.  Note that Form $A(2)$ can appear in one way, while form $A(3)$ can appear in two ways.  We first claim that if $m\geq 3$, then the only constraint on $e$ is $e\geq h_{1,2}$. This inequality is certainly a necessary condition; to see it is sufficient, we note that for form $A(3)$ and beyond there are at least two parallel edges appearing in the face dual to $(1,1)$ (either horizontal or with slope $1$) that are part of the edge $e$, allowing unlimited scaling.  For the form $A(2)$, however, we have that $e=2h_{1,2}$; this can be seen from the dual triangular face.

 \begin{figure}[hbt]
        \centering
        \includegraphics[scale = 1]{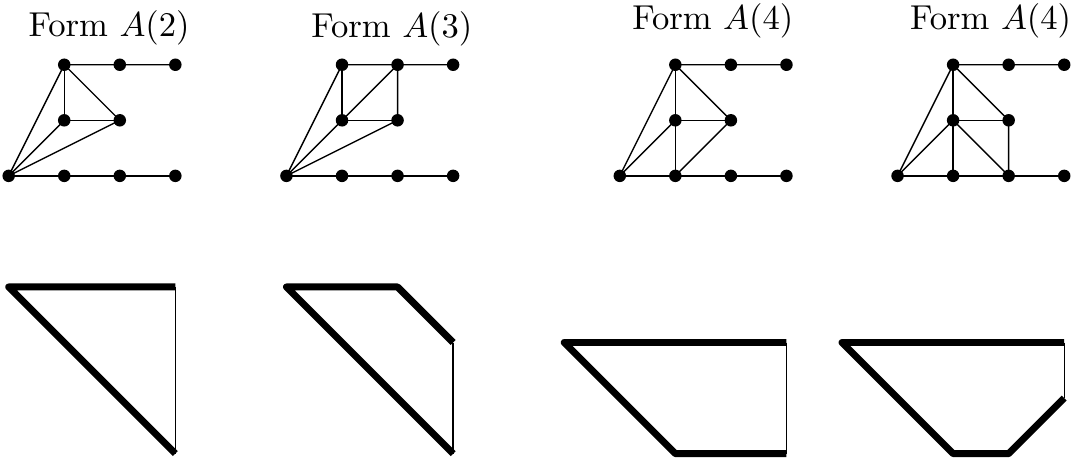}
        \caption{Several forms}
        \label{figure:A_forms}
    \end{figure}
    
Now assume that $P$ has shape $B_k$.  We illustrate several possible forms for $0\leq k\leq 2$ in Figure \ref{figure:B_forms}.  We claim that for $B_k(m)$ with $k\geq 2$, as well as for $B_1(m)$ with $m\geq 4$ and $B_0(m)$ with $m\geq 5$, the only constraint is $e\geq h_{1,2}$.  Again, this holds because we are guaranteed parallel edges contributing to $e$, allowing for unlimited scaling.  For $B_0(3)$ we find $e=h_{1,2}$, and for $B_0(4)$ and $B_1(3)$ we have $h_{1,2}\leq e\leq 2h_{1,2}$.  These again come from considering the illustrated cycles dual to $(1,1)$, and putting more or less length into components of the cycle dual to $(1,1)$.  Thus we have determined the edge constraints on $e$ in all cases.

 \begin{figure}[hbt]
        \centering
        \includegraphics[scale = 0.7]{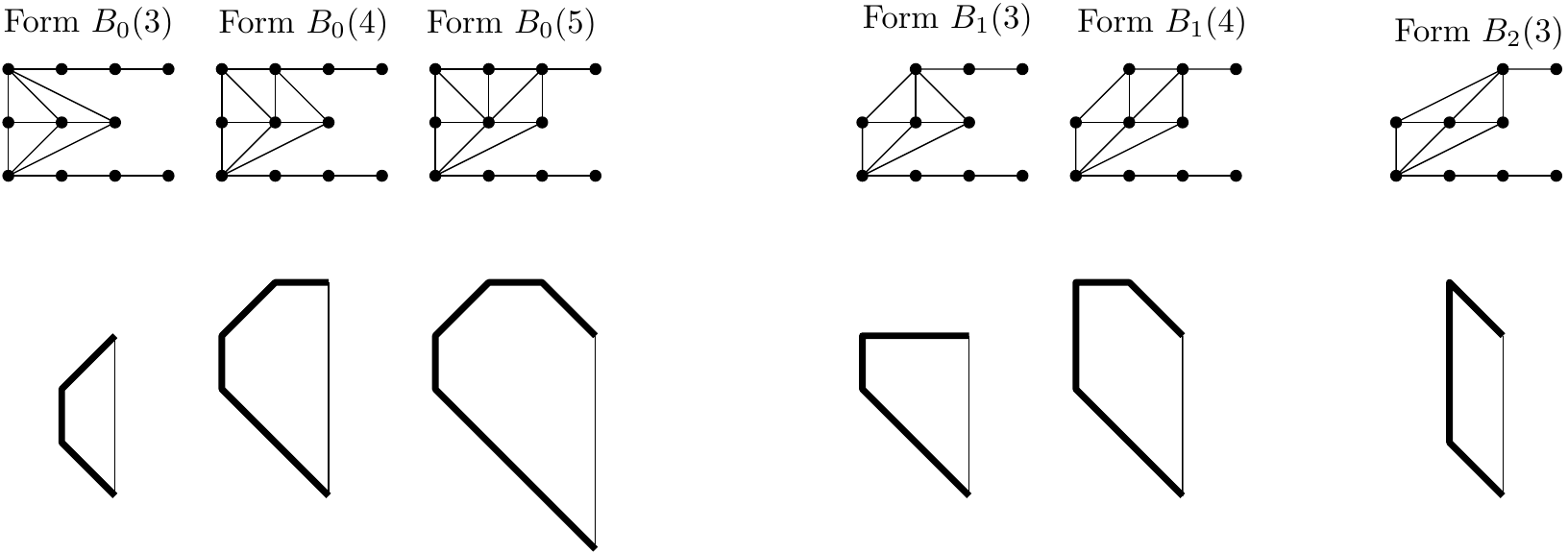}
        \caption{Several forms}
        \label{figure:B_forms}
    \end{figure}
    
A priori we cannot use our shapes to study $f$, since Theorem \ref{theorem:hyperelliptic_classification} differentiates between the left and right sides of the polygon.  However, every right end is equivalent to a left end through a reflection and a shearing transformation, meaning that we may indeed conclude that $f$ satisfies the same constraints as $e$ based on the shape of the polygon at the right end.  %We summarize the edge constraints on $e$ and $f$ in Figure BLAH.

%{\color{blue} Here's a thought that I might try to implement:  maybe we should try to fuse the material about the explicit inequalities?  Technically you can read the dimension off from there.}

\begin{proposition} \label{dimPolyProp}
    Let $\mathcal{T}$ be a unimodular triangulation of a hyperelliptic polygon $P$ polygon of genus $g \geq 2$. Let $b_e$ be the number of end interior points that are connected to another interior lattice point and boundary points that are all collinear. Let $b_m$ be the number of middle interior points connected to exactly $2$ boundary lattice points. Then we have
    \begin{equation*}
    \dim(\mathbb{M}_\mathcal{T}) = 2g-1 - b_e - b_m
    \end{equation*}
    
\end{proposition}

\begin{proof}
To prove this we will consider the requirements on our edge lengths.  There are $3g-3$ edges in a connected trivalent graph of genus $g$.  The conditions that $u_i=\ell_i$ for $2\leq i\leq g-2$ remove $g-2$ degrees of freedom, leaving us with at most $2g-1$ degrees of freedom.

We now ask if there are any other equalities that come out of the length constraints.  For the edges on the $i^{th}$ face with $2\leq i\leq g-1$, the only possibility to have equalities forced are if the upper and lower bounds on $h_{i-i,i}-h_{i,i+1}$ are equal to one another; in this case we lose another degree of freedom.  This happens precisely when $NE(i)+SE(i)=NW(i)+SW(i)$. But since $NE(i)\leq NW(i)$ and $SE(i)\leq SW(i)$ this happens precisely when $NE(i)=NW(i)$ and $SE(i)=SW(i)$.  This is equivalent to $(i,1)$ being connected to exactly one boundary point at height $0$, and exactly one boundary point at height $2$.  This is the same as $(i,1)$ being connected to exactly two boundary points, since it must be connected to at least $2$.  Thus we have $b_m$ more degrees of freedom lost.

Finally, we consider the edges $e$ and $f$.  We lose a degree of freedom in the cases that the form of $\mathcal{T}$ is either $A(2)$ or $B_0(3)$.  But these are precisely the forms where the corresponding interior end point is only connected to collinear boundary points.  This completes the proof.
\end{proof}

Now that we can compute the moduli dimension of a hyperelliptic triangulation, we can move towards finding the moduli dimension of a hyperelliptic polygon.
Since every hyperelliptic polygon of genus $g\geq 2$ is equivalent to exactly one of the polygons from Theorem \ref{theorem:hyperelliptic_classification}, it is enough to find the moduli dimension of each polygon in this classification.

\begin{proposition}\label{prop:all_hyp_dim}
    If $P$ is a hyperelliptic polygon of genus $g$, then
    
    \begin{equation*}
        \dim(\mathbb{M}_P) = \begin{cases} 2g-1 &\mbox{if $P$ is in Class 1} \\
        \min(g+i+j,2g-1) &\mbox{if $P$ is in Class 2(a)}\\
        2g-1 &\mbox{if $P$ is in Class 2(b)}\\
        \min(g+i+j+1,2g-1) &\mbox{if  $P$ is in Class 3(a) and } k \neq 0\\
        \min(g+i+j,2g-1) &\mbox{if $P$ is  in Class 3(a) and } k = 0\\
        \min(g+i+j+1,2g-1) &\mbox{if $P$ is in Class 3(b)},
        \end{cases}
    \end{equation*}
    where $i,j,$ and $k$ are as in the classification.
\end{proposition}

\begin{proof}

First assume that $P$ is in Class $1$, or in Class $2$ with condition (b) satisfied.  Start a triangulation of $P$ as follows.  Connect all interior points in a line segment, and connect the point $(1,2)$ to all interior lattice points.  Then for $1\leq \ell\leq g$, connect the interior lattice point $(\ell,1)$ to the boundary points $(\ell-1,0)$ and $(\ell,0)$; this is possible since $i\geq g$. The start of this triangulation is depicted in Figure \ref{figure:class_1_triangulation} for a Class 1 polygon. Complete this to a unimodular triangulation in any way; this is guaranteed to be regular by \cite{counting_lattice_triangulations}.  Note that both end interior points are connected to at least $3$ not-all-collinear points, and all middle interior points are connected to two points at height $0$.  Thus for this triangulation, $d=e=0$, so $\dim(\mathbb{M}_P)\geq \dim(\mathbb{M}_\mathcal{T})=2g-1$.  Since $\dim(\mathbb{M}_P)\leq 2g-1$ for any hyperelliptic polygon of genus $g$, we conclude that $\dim(\mathbb{M}_P)=2g-1$, as claimed.

\begin{figure}[hbt]
    \centering
    \includegraphics[scale = .2]{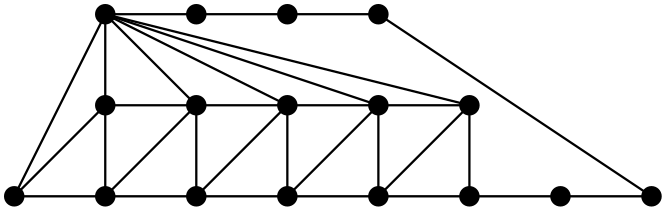}
    \caption{The start of a triangulation of a polygon in Class 1}
    \label{figure:class_1_triangulation}
\end{figure}

Now we assume that $P$ is in Class 2 with condition (a) satisfied;  that is, we assume $0\leq i\leq g$ and $0\leq j\leq i$.  We wish to show that the moduli dimension is $\min(g+i+j,2g-1)$. Certainly we cannot do better than $2g-1$.  Let us argue that we cannot do better than $g+i+j$.  First we remark that in any unimodular triangulation, we connect $(1,1)$ to $(1,2)$ and $(0,0)$, and we connect $(g,1)$ to $(g+1,1)$, $(1+j,2)$, and $(9,0)$.  Thus $(g,1)$ does not contribute to $e$; $(1,1)$ does not contribute to $e$ if and only if it is connected at least two points at height $0$ or $2$; and a middle interior point does not contribute to $d$ if and only if it is connected to at least two points at height $0$ or $2$.  Each time one of the relevant $g-1$ interior points is connected to at least two points at height $0$ or $2$, that uses up one of the edges on the horizontal faces of $P$. There are $i+j$ of these, so at most $\min\{g-1,i+j\}$ of them fail to contribute to $d+e$.  Conversely, at least $g-1-\min\{g-1,i+j\}=\min\{0,g-i-j-1\}$ contribute to $d+e$.  Thus for any triangulation $\mathcal{T}$ we have $\dim(\mathbb{M}_\mathcal{T})=2g-1-e-d\leq 2g-1-\min\{0,g-i-j-1\}=\min\{2g-1,g+i+j\}$.  Note that there does exist a triangulation $\mathcal{T}$ achieving this upper bound: for $1\leq \ell\leq i$, connect the lattice point $(\ell,1)$ to $(\ell-1,0)$ and $(\ell,1)$; and for $i+1\leq\ell\leq \min\{g-1,j+1\}$ connect $(\ell,1)$ to $(\ell-i,2)$.  We conclude that $\dim(\mathbb{M}_P)=\min\{2g-1,g+i+j\}$.  Such a triangulation is illustrated in Figure \ref{figure:class_2a_triangulation}.

\begin{figure}[hbt]
    \centering
    \includegraphics[scale = 1]{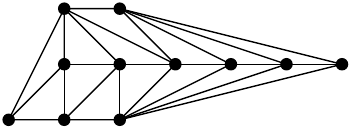}
    \caption{The start of a triangulation of a polygon in Class 2(a)}
    \label{figure:class_2a_triangulation}
\end{figure}

    Finally, we move on to Class 3. In case (a), we have $0 \leq k \leq g+1$ and $0 \leq i \leq g+1−k$ and $0 \leq j \leq i$. If $k \neq 0$, the rightmost and leftmost endpoint interior points never contribute to $e$, so the minimum we can achieve is $g+1$. By the same argument as Class 2(a), we end up getting the maximum possible dimension as $\min(g+1+i+j,2g-1)$. When $k = 0$, the leftmost endpoint interior point can contribute to $e$ like Class 2(a) so we end up with moduli dimension $\min(g+i+j,2g-1)$. For Class 3(b), we wish to show that the moduli dimension is $\min(g+1+i+j,2g-1)$. First, note that the leftmost endpoint interior can contribute to $e$ if and only if $k = 0$. However, when this happens, $i > g+1$, giving us full dimension of $2g-1$. Note that $g + i + j + 1 > 2g-1$ in this case. The rightmost endpoint can contribute to $e$ if and only if $(k+j,2), (g+1,1), (i,0)$ form a line. If they are vertical or the line slopes downward, again $i \geq g+1$ and we have full dimension of $2g-1$. The only other case is if the line slopes upward or $k+j > g+1 >i$. Note that $i +k > g+1$, so $2g-i -2k + 1 < g - k$. Thus, $j < g - k$ and $j+k < g$. Thus, such a line is impossible. With both endpoints not contributing to a loss of degrees of freedom, this now follows the same logic as when $k \neq 0$ in Class 3(a). 
\end{proof}

We now have an easy proof of Theorem \ref{theorem:hyperelliptic}.

\begin{proof}[Proof of Theorem \ref{theorem:hyperelliptic}]
By Proposition \ref{prop:all_hyp_dim} and the fact that $i,j\geq 0$, we have that $g\leq \dim(\mathbb{M}_P)\leq 2g-1$  for any hyperelliptic polygon $P$ of genus $g$.  Conversely, let   $g\leq d\leq 2g-1$.  There does indeed exist a hyperelliptic polygon of genus $g$ with $\dim(\mathbb{M}_P)=d$; for instance, choose $P$ to be in Class 2(a) with $j=0$ and $i=d-g$.
\end{proof}

To our knowledge, there is no known result determining the dimensions of the algebraic moduli spaces \(\mathcal{M}_P\) where \(P\) is a general hyperelliptic polygon, so we cannot compare our formulas to the algebraic case.  An interesting project would be to determine such algebraic dimensions, and to compare them to our tropical ones.

\bibliographystyle{plain}

\end{document}